\documentclass[11pt]{article}
\usepackage{amsmath}
\usepackage{amssymb}
\usepackage{amsthm}
\usepackage{dsfont}
\usepackage{enumerate}
\usepackage[utf8]{inputenc}
\usepackage{authblk}
\usepackage[pdftex,colorlinks,%
           bookmarks=true,%
           pdftitle=Viral\ Marketing\ On\ Configuration\ Model,%
           pdfauthor=B.\ Blaszczyszyn%
          \ -\  K.\ Gaurav]{hyperref}
\hypersetup{
   bookmarksnumbered,
   pdfstartview={FitH},
   citecolor={blue},
   linkcolor={red},
   urlcolor=[rgb]{0,0.55,0},
   pdfpagemode={UseOutlines}
}

\usepackage[english]{babel}
\newtheorem{thm}{Theorem}[section]
\newtheorem{cor}[thm]{Corollary}
\newtheorem{cond}[thm]{Condition}
\newtheorem{lem}[thm]{Lemma}
\newtheorem{prop}[thm]{Proposition}
\newtheorem{remark}[thm]{Remark}
\newenvironment{rem}[1][]{\begin{remark}[#1]\rm}{\end{remark}}
\newtheorem{defin}[thm]{Definition}

\newtheorem{example}[thm]{ Example}

\newtheorem{fact}[thm]{Fact}

\begin{document}

\title{Viral Marketing On Configuration Model}

\author{Bart{\l}omiej B{\l}aszczyszyn\thanks{Inria/ENS,
23 av. d'Italie 75214 Paris, France;  Bartek.Blaszczyszyn@ens.fr}
and Kumar Gaurav\thanks{UPMC/Inria, 23 av. d'Italie 75214 Paris, France;
Kumar.Gaurav@inria.fr}}

\maketitle

\begin{abstract}
We consider propagation of influence on a Configuration Model, where each vertex can be influenced by any of its neighbours but in its turn, it can only influence a random subset of its neighbours.  Our (enhanced) model is described by the total degree of the typical vertex, representing the total number of its neighbours and the transmitter degree, representing the number of neighbours it is able to influence.  We give a condition involving the joint distribution of these two degrees, which if satisfied would allow with high probability the influence to reach a non-negligible fraction of the vertices, called a {\em big (influenced) component}, provided that the source vertex is chosen from a set of \textit{good pioneers}. We show that asymptotically the big component is essentially the same, regardless of the good pioneer we choose, and we explicitly evaluate the asymptotic relative size of this component.  Finally, under some additional technical assumption we calculate the relative size of the set of good pioneers.  The main technical tool employed is the ``fluid limit'' analysis of the joint exploration of the configuration model and the propagation of the influence up to the time when a big influenced component is completed. This method was introduced in Janson~\& Luczak~(2008) to study the giant component of the configu\-ration model. Using this approach we study also a reverse dynamic, which traces all the possible sources of influence of a given vertex, and which by a new ``duality'' relation allows to characterise the set of good pioneers.
\end{abstract}


\section{Introduction}
The desire for understanding the mechanics of complex
networks~\cite{albert2002statistical,newman2003structure}, describing a wide range of systems in nature and society,
motivated many applied and theoretical investigations of the last two decades.
A motivation for our work can come from the phenomenon of viral
marketing in social networks: A person after getting acquainted with an advertisement (or a news article or a Gangnam style video, for that matter) through one of his
 ``friends'', may decide to share it with some (not necessarily all) 
of his friends, who will, in turn, pass it along to some of their
friends, and so on. The campaign is successful if 
starting from a relatively small number of initially targeted
persons, the influence (or information) can spread as an epidemic
``infecting'' a 
non-negligible fraction of the population.

\paragraph{Enhanced Configuration Model}
Traditionally, social networks have been modeled as random
graphs~\cite{durrett2007random,Remco},
where the vertices denote the 
individuals and edges connect individuals who know one another.
The Configuration Model is considered as a useful approximation in
this matter, 
and we assume it for our study of the viral
marketing.  It is a random (multi-)graph, whose vertices have
prescribed degrees, realized  by  half-edges emanating from them
and  uniformly pair-wise matched to each other to create edges.
In order to model a selective character of the influence propagation 
(each vertex can be influenced by any of its
neighbours but in its turn, it can only influence a subset of its
neighbours),
we  enhance the original Configuration
  Model  by 
considering {\em two types of half-edges}. 
{\em Transmitter half edges} of a given vertex represent links through which 
this vertex will influence (pass the information once it has it) 
to its neighbours. Its {\em receiver
half-edges} represent links through which this vertex 
will not propagate the information to its neighbours.
The neighbours  receive the information both through their transmitter and
receiver half-edges matched to a transmitter half edge of the
information sender. 
The two types of half-edges are not distinguished
during the uniform pair-wise matching of all half-edges, 
but only to trace the propagation of information.
Assuming the usual consistency conditions for the numbers 
of transmitter and receiver half-edges,
the Enhanced Configuration Model is asymptotically (when the
number of vertices $n$ goes to infinity)  described by the vector  of two, 
not necessarily independent, integer valued random variables,
representing the {\em transmitter} and {\em receiver degree} of the
typical vertex.  Equivalently, we can consider the {\em total vertex
degree},  representing the total number of friends of a 
person and its transmitter degree, representing 
the number of friends he/she is able to influence.

\paragraph{Results}
We consider the advertisement campaign started from some 
initial target (source vertex) and following
the aforementioned dynamic on a realization of the  Enhanced Configuration
Model of the total number of vertices  $n$.
The results are formulated with high probability (whp), i.e.  
with probability approaching one as $n\to\infty$.

First, we give a condition involving the total degree 
and the transmitter degree distributions of the Enhanced Configuration
Model,  which if satisfied, would allow   whp the advertisement campaign to
reach a non-negligible fraction ($O(n)$) of the
population, called a {\em big (influenced) component}, 
provided   that the initial
target  is chosen from a set
of \textit{good pioneers}. Further in this case, we show that 
asymptotically the {\em big component 
is essentially the same} 
regardless of the good pioneer chosen, and we
explicitly evaluate the asymptotic size of this component relative
to~$n$.  The essential uniqueness of the big component means 
that the subsets of influenced vertices reached from two different
good pioneers  differ by at most $o(n)$
vertices  whp.   Finally,  under some additional technical assumption  we calculate the relative size of the set of good pioneers.

\paragraph{Methodology}
A standard technique for the analysis of diffusion of information on the
Configuration Model consists in 
simultaneous exploration of the model and the propagation of the influence.
We adopt this technique and, more precisely, 
the approach  proposed in \cite{JanLuc} for the study of the
giant component of the (classical) Configuration Model. 
In this approach, instead of the branching process
approximating the early stages of the graph exploration,  one uses  
a ``fluid limit'' analysis of the process up to the time when 
the exploration of the big component is completed. We tailor this
method to our specific dynamic of influence propagation
and calculate the relative size of the big influenced component, as well as
prove its essential uniqueness.

A fundamental difference with respect to the study of the
giant component of the classical model stems from the directional
character of our propagation dynamic. Precisely, 
the edges matching a transmitter and a receiver half-edge
can relay the influence from the transmitter half-edge to the receiver one,
but not the other way around. This means that  the good pioneers do not need
to belong to the big (influenced) component, and vice versa. 
In this context,  we introduce a {\em reverse dynamic}, in which 
a message (think of an ``acknowledgement'') can be sent in the
reversed direction on every edge (from an arbitrary half-edge to the
receiver one),  which traces all the possible sources of influence of
a given vertex.  
This reversed dynamic can be studied using the same approach as the
original one. In particular, one can establish the essential uniqueness
of the big component of the reversed process as well as calculate its
relative size. Interestingly, this relative size coincides with the
probability of the non-extinction of the branching process
approximating the initial phase of the original exploration process,
whence the hypothesis that the  big component of the reverse process
coincides with the set of good pioneers. We  prove this
conjecture under some additional (technical) assumption.
We believe the method of introducing a reverse process
to derive results for the original one has not
been seen in a related context in the existing literature.

\paragraph{Related Work}
The propagation of influence through a network has been previously
studied in various contexts. The Configuration Model 
has formed the base for an increasing number of influence propagation
studies, of which one  relevant to the phenomenon of viral networking in social networks is discussed in \cite{Marcam} and \cite{Marc2}, where a vertex in the network gets influenced 
only if a certain proportion of its neighbours have already been influenced. This interesting propagation dynamic is further studied by introducing 
\textit{cliques} in Configuration Model to observe the impact of clustering on the size of the population influenced (see \cite{Marc3},\cite{Marc4}).  
This dynamic is a kind of \textit{pull model} where influence propagation depends on whether a vertex decides to receive the influence
from its neighbours.  We study a \textit{push model}, where 
the influence propagation depends on whether a vertex decides to
transmit the influence. A propagation dynamic where 
every influenced node, at all times, keeps choosing  one of its neighbours
uniformly at random and transmits the message to it is studied on a
$d$-regular graph in \cite{Nikola}. This dynamic is close in its
spirit to the one we considered in this paper, however the process stops
when {\em all} nodes receive the message, and this stopping time is studied in
the paper. The same dynamic but restricted to some (possibly random)
maximal number of transmissions allowed for each vertex 
is considered  in \cite{Comets}  on a complete graph. 
This can be thought as a special case of our dynamic (although we
study it on a different underlying graph) if we assume that the 
transmitter and receiver degrees correspond to  the number of 
collected and non-collected coupons, respectively, 
in the classical coupon collector  
problem with the number of coupons being the  vertex degree and the
number of trials being the number of allowed transmissions.
In a more applied context, a rudimentary special case of our dynamic 
of influence propagation has actually been studied on real-world
networks like \textit{flixster} 
and \textit{flickr} (see \cite{Goyal}).

\paragraph{Paper organization}
The remaining part of this paper is organized as follows.
In the next section we describe our model and formulate the results. In Sections~\ref{s.Forward} and~\ref{s.Reverse}
we analyze, respectively,  the original and reversed dynamic of influence propagation.   The relations between the two dynamics are explored in Section~\ref{s.Duality}.

\section{Notation and Results}    
\label{s.Results}
Given a degree sequence $(d_i^{(n)})_1^n$ for $n$ vertices labelled $1$ to $n$, Configuration Model, denoted $G^*(n,(d_i)_1^n)$, 
is a random multigraph obtained by giving $d_i$ half-edges to each vertex $i$ and then uniformly matching pair-wise the set of 
half-edges. Conditioning the Configuration Model to be simple, we obtain a uniform random graph with the given degree sequence, 
denoted by $G(n,(d_i)_1^n)$. Since it is convenient to work with the Configuration Model, we will prove all our results for the 
Configuration Model and the corresponding results for the uniform random graph can be obtained by passing through a standard 
conditioning procedure (see, for example, \cite{Remco}).

Further, in our model, we represent the degree, $d_i$, of each vertex $i$ as the sum of two (not necessarily independent) 
degrees: transmitter degree, $d_i^{(t)}$ and receiver degree, $d_i^{(r)}$.

We will asssume the following set of consistency conditions for our enhanced  Configuration Model, which are analogous to 
those assumed for Configuration Model in \cite{JanLuc}.

\begin{cond}
\label{degden}
For each $n$, $\text{\bf d}^{(n)}=(d_i)_1^n$,is a sequence of non-negative integers such that $\sum_{i=1}^nd_i:=2m$ is even 
and for each $i$, $d_i=d_i^{(r)}+d_i^{(t)}$. 
For $k \in \mathbb{N}$, let $u_{k,l}=\vert \{i:d_i^{(r)}=k, d_i^{(t)}=l\} \vert$, and $D_n^{(r)}$ and $D_n^{(t)}$ be the
receiver and transmitter degrees respectively of a uniformly chosen vertex in our model, i.e., 
$\mathbb{P}(D_n^{(r)}=k, D_n^{(t)}=l)=u_{k,l}/n$. Let $D^{(r)}$ and $D^{(t)}$
be two random variables taking value in non-negative integers with joint probability distribution 
$(p_{v,w})_{(v,w)\in \mathbb{N}^2}$, and $D := D^{(r)} + D^{(t)}$.
Then the following hold.
\begin{enumerate}[(i)]
 \item $\frac{u_{k,l}}{n} \to p_{k,l}$ for all $(k,l) \in \mathbb{N}^2$. \label{a1}
 \item $\mathbb{E}[D] = \mathbb{E}[D^{(r)}+D^{(t)}] = \sum_{k,l}(k+l)p_{k,l} \in (0,\infty)$. Let $\lambda_r = \mathbb{E}[D^{(r)}]$, 
$\lambda_t = \mathbb{E}[D^{(t)}]$ and $\lambda=\lambda_r+\lambda_t$. \label{a2}
 \item $\sum_{i=1}^n(d_i)^2= O(n)$. \label{a3}
 \item $\mathbb{P}(D=1)>0$. \label{a4}
\end{enumerate}
\end{cond}

Let $g(x,y):=\mathbb{E}[x^{D^{(r)}}y^{D^{(t)}}]$ be the joint probability generating function of $(p_{v,w})_{(v,w)\in \mathbb{N}^2}$. 
Further let
\begin{equation}
 h(x):=x\left.\frac{\partial g(x,y)}{\partial y} \right|_{y=x} = \mathbb{E}[D^{(t)}x^D],
\end{equation}
and
\begin{equation}
\label{hx}
 H(x):=\lambda x^2-\lambda_r x-h(x).
\end{equation}

If two neighbouring vertices $x$ and $y$ are connected via the pairing of a transmitter half-edge of $x$ with any half-edge of $y$, 
then $x$ has the ability to directly influence $y$. More generally, for any two vertices $x$ and $y$ in the 
graph and $k\geq 1$, if there exists a set of vertices $x_0=x,x_1,.....,x_{k-1},x_k=y$ such that $\forall i:1 \leq i \leq k$, $x_{i-1}$ 
has the ability to directly influence $x_{i}$, we say that $x$ has the
ability to influence $y$ and denote it by $x \to y$; 
in other words, $y$ can be influenced starting from the initial source $x$. Let $C(x)$ be the set of vertices of $G(n,(d_i)_1^n)$ which are influenced starting 
from an initial source of influence, $x$, 
until the process stops, i.e.,
\begin{equation}
C(x)=\left\{y \in v(G(n,(d_i)_1^n)) : x\to y\right\},
\end{equation}
where $v(G(n,(d_i)_1^n))$ denotes the set of all the vertices of $G(n,(d_i)_1^n)$. We use $\left| . \right|$ to denote the number of 
elements in a set here, although at other times we also use the symbol to denote the absolute value, which would be clear from the context.
We have the following theorems for the forward influence propagation process.

\begin{thm}
\label{inf1out}
Suppose that Condition \ref{degden} holds and consider the random graph $G(n,(d_i)_1^n)$, letting $n \to \infty$.

If $\mathbb{E}[D^{(t)}D] > \mathbb{E}[D^{(t)}+D]$, then there is a unique $\xi \in (0,1)$ such that $H(\xi)=0$ and 
there exists at least one $x_n$ in $G(n,(d_i)_1^n)$ such that 
\begin{equation}
 \frac{\left|C(x_n)\right|}{n} \xrightarrow{p} 1-g(\xi,\xi) > 0.
\end{equation} 

\end{thm}

We denote $C(x_n)$ constructed in the proof of Theorem \ref{inf1out} by $C^*$. For every $\epsilon>0$, 
let 
\begin{equation*}
 \mathbb{C}^s(\epsilon):=\left\{x \in v(G(n,(d_i)_1^n)): \left|C(x)\right|/n < \epsilon\right\}
\end{equation*}
and 
\begin{equation*}
 \mathbb{C}^L(\epsilon):=\left\{x \in v(G(n,(d_i)_1^n)): \left|C(x)\vartriangle C^* \right|/n < \epsilon\right\},
\end{equation*}
where $\vartriangle$ 
denotes the symmetric difference. 

\begin{thm}
\label{inf2out}


Under assumptions of Theorem \ref{inf1out}, we have that
\begin{equation}
 \forall \epsilon, \quad  \frac{\left|\mathbb{C}^s(\epsilon)\right|+\left|\mathbb{C}^L(\epsilon)\right|}{n} \xrightarrow{p} 1.
\end{equation}
\end{thm}

Informally, the above theorem says that asymptotically ($n\to \infty$) and under assumptions of Theorem \ref{inf1out}, there 
is essentially one and only one big (i.e., of size $O(n)$) graph component 
that can possibly be influenced starting with propagation from a given vertex in the graph. What this theorem
doesn't tell, however, is the relative size of the set of vertices which are indeed able to reach this big component (we call them
\textit{pioneers}) to the set of vertices which are able to reach only a component of size $o(n)$, and this is the question we turn to next. 

Our analysis technique to obtain the above results involves the simultaneous exploration of the Configuration Model and 
the propagation of influence. Another commonly used method to explore the components of Configuration Model is to make the 
branching process approximation in the 
initial stages of the exploration process. Although we won't explicitly follow this path in this paper, an heuristic analysis of 
the branching process approximation of our propagation model provides some important insights about the size of the set of pioneers.

We will need the following fundamental result on branching processes (see, for example, \cite{Mass}).

\begin{fact}[Survival vs. Extiction]
\label{Galt}
 For the Galton-Watson branching process whose progeny distribution is given by a random variable $Z$, the extinction 
probability $p_{ext}$ is given by the smallest solution in $[0,1]$ of 
\begin{equation}
\label{extn}
 x=\mathbb{E}(x^Z).
\end{equation}
In particular, the following regimes can happen:
\begin{enumerate}[(i)]
 \item Subcritical regime: If $\mathbb{E}[Z]<1$, then $p_{ext}=1$. \label{Galt1}
 \item Critical regime: If $\mathbb{E}[Z]=1$ and $Z$ is not deterministic, then $p_{ext}=1$. \label{Galt2}
 \item Supercritical regime: If $\mathbb{E}[Z]>1$, then $p_{ext}<1$.  \label{Galt3}
\end{enumerate}

\end{fact}

Now coming to the approximation, if we start the exploration with a uniformly chosen vertex $i$, then the number of 
its neighbours that it does not influence and those that it does, denoted by the random vector $(D_i^{(r)},D_i^{(t)})$, 
will have a joint distribution $(p_{v,w})$. But since the probability of getting influenced is proportional to the degree, 
the number of neighbours of a first-generation vertex excluding its parent (the vertex which influenced it) won't follow this 
joint distribution. Their joint distribution 
as well the joint distribution in the subsequent generations, denoted by $(\widetilde{D}^{(r)},\widetilde{D}^{(t)})$, 
is given by
\begin{equation}
{\widetilde{p}}_{v,w}=\frac{\left(v+1\right){p}_{v+1,w}+\left(w+1\right){p}_{v,w+1}}{\lambda }.
\end{equation}

Note that Condition \ref{degden}(\ref{a4}) implies that $\mathbb{P}(\widetilde{D}^{(t)}=0)>0$, and therefore, from Fact \ref{Galt}, 
this branching process gets extinct a.s. unless, 

\begin{align*}
&\mathbb{E}\left[{\widetilde{D}}^{\left(t\right)}\right]>1; \\
\text{equivalently,}\quad &\sum _{v,w}w{\widetilde{p}}_{v,w}>1, \\
 &\sum _{v,w}\frac{w\left(v+1\right){p}_{v+1,w}+w\left(w+1\right){p}_{v,w+1}}{\lambda }>1, \\
&\mathbb{E}\left[{D}^{\left(r\right)}{D}^{\left(t\right)}\right]+
\mathbb{E}\left[{D}^{\left(t\right)}\left({D}^{\left(t\right)}-1\right)\right]>\mathbb{E}\left[D\right], \\
&\mathbb{E}\left[D{D}^{\left(t\right)}\right]> \mathbb{E}\left[D+{D}^{\left(t\right)}\right].
\end{align*}

This condition for non-extinction of branching process remarkably agrees with the condition in Theorem \ref{inf1out}
which determines the possibility of influencing a non-negligible proportion of population.

Further from Fact \ref{Galt}, if this condition is satisfied, the extinction probability of the branching process which 
diverges from the first-generation vertex, $\widetilde{p}_{ext}$, is given by the smallest $x\in (0,1)$ which satisfies
\begin{align}
\label{bra}
&\mathbb{E}\left[{x}^{\widetilde{D}^{\left(t\right)}}\right]=x; \nonumber \\
\text{equivalently,}\quad &\sum _{v,w}\frac{{x}^{w}\left(v+1\right){p}_{v+1,w}+\left(w+1\right){x}^{w}{p}_{v,w+1}}{\lambda }=x, \nonumber \\
&\mathbb{E}\left[{D}^{\left(r\right)}{x}^{{D}^{\left(t\right)}}\right]
+\mathbb{E}\left[{D}^{\left(t\right)}{x}^{{D}^{\left(t\right)}-1}\right]
=x \mathbb{E}\left[D\right], \nonumber \\
&\mathbb{E}\left[D\right]{x}^{2}-\mathbb{E}\left[{D}^{\left(t\right)}{x}^{{D}^{\left(t\right)}}\right]
-x\mathbb{E}\left[{D}^{\left(r\right)}{x}^{{D}^{\left(t\right)}}\right]=0.
\end{align}

Note that $0$ is excluded as a solution since $\mathbb{P}(\widetilde{D}^{(t)}=0)>0$. 

Finally, the extinction probability of the branching process starting from the root, ${p}_{ext}$, is given by
\begin{equation}
{p}_{ext}=\mathbb{E}\left[({\widetilde{p}_{ext}})^{{D}^{\left(t\right)}}\right].
\end{equation}

Since the root is uniformly chosen, we would expect the proportion of the vertices which can influence a non-negligible proportion
 to be roughly $1-{p}_{ext}=1-\mathbb{E}\left[({\widetilde{p}_{ext}})^{{D}^{\left(t\right)}}\right]$.
Indeed, we confirm this result using a more rigorous analysis involving the introduction and study of a reverse influence 
propagation which essentially traces all the possible sources of influence of a given vertex. This method of introducing a reverse 
process (in a way, dual to the original process)to derive results for the original process has not been seen in a related context 
in the existing literature to the best of our knowledge, although the analysis of this dual process uses the familiar tools used for 
the original process.

Let $\overline{g}(x):=\mathbb{E}[x^{D^{(t)}}]$, $\overline{h}(x):=\mathbb{E}[D^{(t)}x^{D^{(t)}}]+x\mathbb{E}[D^{(r)}x^{D^{(t)}}]$ and
\begin{equation}
\label{Hx2}
 \overline{H}(x):=\mathbb{E}[D] x^2-\overline{h}(x)=\lambda x^2-\overline{h}(x).
\end{equation}

Let $\overline{C}(y)$ be the set of vertices of $G(n,(d_i)_1^n)$ starting from which $y$ can be influenced, i.e., 
$\overline{C}(y):=\left\{x \in v(G(n,(d_i)_1^n)) : x\to y\right\}$.
We have the following theorems for the dual backward propagation process.

\begin{thm}
\label{inf1bar}
Under assumptions of Theorem \ref{inf1out}, 
there is a unique $\overline{\xi} \in (0,1)$ such that $\overline{H}(\overline{\xi})=0$ and there exists at least one 
$y_n$ in $G^*(n,(d_i)_1^n)$ such that 
\begin{equation}
 \frac{\left|\overline{C}(y_n)\right|}{n} \xrightarrow{p} 1-\overline{g}(\overline{\xi}) > 0.
\end{equation} 
\end{thm}
Remark that $\overline{H}(x)=0$ is the same as equation (\ref{bra}) and therefore 
$\overline{\xi} \equiv \widetilde{p}_{ext}$ and $1-\overline{g}(\overline{\xi})\equiv p_{ext}$ 
from the branching process approximation.

We denote $\overline{C}(y_n)$ constructed in the proof of Theorem \ref{inf1bar} by $\overline{C}^*$. 
For every $\epsilon>0$, let 
\begin{equation*}
 \overline{\mathbb{C}}^s(\epsilon):=\left\{y \in v(G(n,(d_i)_1^n)): \left|\overline{C}(y)\right|/n < \epsilon\right\},
\end{equation*}
and 
\begin{equation*}
 \overline{\mathbb{C}}^L(\epsilon):=\left\{y \in v(G(n,(d_i)_1^n)): \left|\overline{C}(y)\vartriangle \overline{C}^* \right|/n 
< \epsilon\right\}.
\end{equation*}

\begin{thm}
\label{inf2bar}
Under assumptions of Theorem \ref{inf1out},
\begin{equation}
 \forall \epsilon, \quad  \frac{\left|\overline{\mathbb{C}}^s(\epsilon)\right|+\left|\overline{\mathbb{C}}^L(\epsilon)\right|}{n} 
 \xrightarrow{p} 1.
\end{equation}
\end{thm} 

Informally, the above theorem says that asymptotically ($n\to \infty$) and under assumptions of Theorem \ref{inf1out}, there 
is essentially one and only one big \textit{source} component in the graph, to which a given vertex can possibly trace back while tracing 
all the possible sources of its influence.

Finally, we have the following theorem which establishes the duality relation between the two processes.

\begin{thm}
\label{dut}
Under assumptions of Theorem \ref{inf1out}, for any $\epsilon>0$ and $n\to \infty$,
\begin{equation}
n^{-1}|\overline{\mathbb{C}}^L(\epsilon)|\left| n^{-1} |\overline{C}^*|-n^{-1}|\mathbb{C}^L(\epsilon)|\right|
 \leq \alpha \epsilon + R_n(\epsilon), 
\end{equation}
where $\alpha>0$ and $R_n(\epsilon) \xrightarrow{p} 0$.
\end{thm}

The theorem leads to the following fundamental result of this paper, where it all comes together and we are able to 
essentially identify, under one additional assumption apart from those in Theorem \ref{inf1out}, the set of pioneers with the one 
big source component that we discovered above. In particular, this gives us the relative size (w.r.t. $n$) of the set of pioneers 
since we know the relative size of the source component.

\begin{cor}
\label{du}
Under assumptions of Theorem \ref{inf1out}, for any $\epsilon>0$ and $n\to \infty$, if there exists $a>0$ such that 
$n^{-1}|\mathbb{C}^L(\epsilon)|>a$ whp, then
\begin{equation}
\label{fin}
 n^{-1} |\mathbb{C}^L(\epsilon) \vartriangle \overline{C}^*| \leq \alpha' \epsilon + R_n'(\epsilon), 
\end{equation}
where $\alpha'>0$ and $R_n'(\epsilon) \xrightarrow{p} 0$.
\end{cor}

\begin{rem}
In particular, if $\mathbb{E}[D^{(t)}(D^{(t)}-2)]>0$, then the Configuration Model with the degree sequence $(d_i^{(t)})_1^n$ will have 
a giant component $C^{(t)}$ whp. In this case, whp $n^{-1}|\mathbb{C}^L(\epsilon)| \geq n^{-1}|C^{(t)}| > a$ for some $a>0$, and thus the 
condition in 
the above corollary is satisfied.
\end{rem}

\paragraph{Future Work}

There is a strong indication that in Corollary \ref{du}, we do not need the lower bound on $n^{-1}|\mathbb{C}^L(\epsilon)|$ for (\ref{fin}) to 
hold. One possible approach to prove this would be to make rigorous the branching process approximation heuristically illustrated in the previous 
section to provide insight (see \cite{Britt}, where the branching process approximation is used to find the largest component of 
Erd\"{o}s-R\'{e}nyi graph). This approach could give not only the required lower bound on $n^{-1}|\mathbb{C}^L(\epsilon)|$ in Corollary \ref{du}, but 
even the desired approximation of $n^{-1}|\mathbb{C}^L(\epsilon)|$ which we otherwise obtain by the identification of $\mathbb{C}^L(\epsilon)$ with 
$\overline{C}^*$ in Corollary \ref{du}. But even in that case, the introduction of the dual process which leads to the identification of 
$\mathbb{C}^L(\epsilon)$ with $\overline{C}^*$ is useful since this would provide us with important additional information regarding the 
structure of $\mathbb{C}^L(\epsilon)$, which  we have not explored 
in this paper. 

We also believe that the sufficient condition on the total and the transmitter degree distribution ($\mathbb{E}[D^{(t)}D] > \mathbb{E}[D^{(t)}+D]$) 
in Theorem \ref{inf1out} for influence 
propagation to go viral, is necessary as well. 

\section{Analysis of the Original Forward-Propagation Process}
 \label{s.Forward}
 The following analysis is similar to the one presented in \cite{JanLuc} and wherever the proofs of analogous lemmas, theorems etc. 
 don't have any 
 new point of note, we refer the reader to \cite{JanLuc} without giving the proofs. 
 
 Throughout the construction and propagation process, we keep track of what we call \textit{active transmitter} half-edges. To begin 
 with, 
 all 
 the vertices and the attached half-edges are \textit{sleeping} but once influenced, a vertex and its half-edges become \textit{active}. 
 Both 
 sleeping and active half-edges at any time constitute what we call \textit{living} half-edges and when two half-edges are matched to 
 reveal an 
 edge along which the flow of influence has occurred, the half-edges are pronounced \textit{dead}. Half-edges are further classified 
 according to 
 their ability or inability to transmit information as \textit{transmitters} and \textit{receivers} respectively. 
 We initially give all the half-edges i.i.d. random maximal lifetimes with distribution 
 given by $\tau \sim \text{exp}(1)$, then go through the following algorithm.
 
 \begin{enumerate}[C1]
  \item If there is no active half-edge (as in the beginning), select a sleeping vertex and declare it active, along with all its 
  half-edges. 
 For definiteness, we choose the vertex uniformly at random among all sleeping vertices. If there is no sleeping 
 vertex left, the process stops. \label{c1}
  \item Pick an active transmitter half-edge and kill it. \label{c2}
  \item Wait until the next living half-edge dies (spontaneously, due to the expiration of its exponential life-time). This is joined to 
  the one killed in previous step to form an edge of the graph along 
 which information has been transmitted. If the vertex it belongs to is sleeping, we change its status to active, along with all of its 
 half-edges. 
 Repeat from the first step. \label{c3}
 \end{enumerate}
 
 Every time C\ref{c1} is performed, we choose a vertex and trace the flow of influence from here onwards. Just before C\ref{c1} is 
 performed again, 
 when the number of active transmitter half-edges goes to $0$, we've explored the extent of the graph component that the chosen vertex 
 can influence, that had not been previously influenced.
 
 Let $S_T(t)$, $S_R(t)$, $A_T(t)$ and $A_R(t)$ represent the number of sleeping transmitter, sleeping receiver, active transmitter and 
 active receiver half-edges, respectively, at time $t$. Therefore, $R(t):=A_R(t)+S_R(t)$ and 
 $L(t):=A_T(t)+A_R(t)+S_T(t)+S_R(t)=A_T(t)+S_T(t)+R(t)$ 
 denotes the number of receiver and living half-edges, respectively, at time $t$. 
 
 For definiteness, we will take them all to be right-continuous, which along with C\ref{c1} entails that $L(0)=2m-1$. Subsequently, 
 whenever 
 a living half-edge dies spontaneously, C\ref{c3} is performed, immediately followed by C\ref{c2}. As such, $L(t)$ is decreased by 
 2 every 
 time a living half-edge dies spontaneously, up until the last living one die and the process terminates.
 Also remark that all the receiver half-edges, both sleeping and active, continue to die spontaneously.
 
 The following consequences of \textit{Glivenko-Cantelli} theorem are analogous to those given in \cite{JanLuc} and we state them 
 without proof.
  
 \begin{lem}
 \label{lem1}
  As $n\to \infty$ ,
 \begin{equation}
  \sup_{t\geq 0}\left| n^{-1}L(t)-\lambda e^{-2t}\right| \xrightarrow{p} 0 .
 \end{equation}
 \end{lem}

 \begin{lem}
 \label{lem2}
 As $n\to \infty$ ,
 \begin{equation}
  \sup_{t\geq 0}\left| n^{-1}R(t)-\lambda_r e^{-t}\right| \xrightarrow{p} 0 .
 \end{equation} 
 \end{lem} 
 
 Let $V_{k,l}(t)$ be the number of sleeping vertices at time $t$ which started with receiver and transmitter degrees $k$ and $l$ 
 respectively . Clearly,
 \begin{equation}
  S_T(t)=\sum_{k,l}l V_{k,l}(t).
 \end{equation}
 Among the three steps, only C\ref{c1} is responsible for premature death (before the expiration of exponential life-time) of sleeping 
 vertices. We first ignore its effect by letting $\widetilde{V}_{k,l}(t)$ be 
 the number of vertices with receiver and transmitter degrees $k$ and $l$ respectively, such that all their half-edges would die 
 spontaneously 
 (without the aid of C\ref{c1}) after time $t$. Correspondingly, let $\widetilde{S}_T(t)=\sum_{k,l}l \widetilde{V}_{k,l}(t)$.
 
 Then,
 \begin{lem}
 \label {lem3}
   As $n\to \infty$ ,
 \begin{equation}
 \label{lem31}
  \sup_{t\geq 0}\left| n^{-1}\widetilde{V}_{k,l}(t)-p_{k,l}e^{-(k+l)t})\right|\xrightarrow{p} 0 .
 \end{equation}
 for all $(k,l) \in \mathbb{N}^2$, and 
 \begin{equation}
 \label{lem32}
  \sup_{t\geq 0}\left| n^{-1} \sum_{k,l}\widetilde{V}_{k,l}(t)-g(e^{-t},e^{-t})\right| \xrightarrow{p} 0 .
 \end{equation}
 \begin{equation}
 \label{lem33}
  \sup_{t\geq 0}\left|n^{-1}\widetilde{S}_T(t)-h(e^{-t})\right| \xrightarrow{p} 0 .
 \end{equation} 
 \end{lem}
 
 \begin{proof}
  Again, (\ref{lem31}) follows from \textit{Glivenko-Cantelli} theorem. 
 To prove (\ref{lem33}), note that by Condition \ref{degden}(\ref{a3}), $D_n=D_n^{(r)}+D_n^{(t)}$ are uniformly integrable, i.e., for 
 every $\epsilon>0$ there exists $K<\infty$ such that
 for all $n$, 
 \begin{equation}
  \mathbb{E}(D_n;D_n>K) = \sum_{(k,l;k+l>K)}(k+l)\frac{u_{k,l}}{n}<\epsilon .
 \end{equation}
 This, by Fatou's inequality, further implies that 
 \begin{equation}
  \sum_{(k,l;k+l>K)}(k+l)p_{k,l}<\epsilon .
 \end{equation}
 Thus, by (\ref{lem31}), we have whp,
 \begin{align*}
  \sup_{t\geq 0}\left| n^{-1}\widetilde{S}_T(t)-h(e^{-t})\right| &=  \sup_{t\geq 0}\left| \sum_{k,l} l(n^{-1} 
 \widetilde{V}_{k,l}(t)-p_{k,l}e^{-(k+l)t})\right| \\
 &\leq \sum_{(k,l;k+l \leq K)} l \sup_{t\geq 0}\left| (n^{-1} \widetilde{V}_{k,l}(t)-p_{k,l}e^{-(k+l)t})\right| + \\
 &\sum_{(k,l;k+l>K)}l(\frac{u_{k,l}}{n}+p_{k,l}) \\
 &\leq \epsilon +\epsilon + \epsilon,
 \end{align*}
 which proves (\ref{lem33}). A similar argument also proves (\ref{lem32}).
 \end{proof}
 
 \begin{lem}
 \label{lem4}
  If $d_{max} := \max_id_i$ is the maximum degree of $G^*(n,(d_i)_1^n)$, then
 \begin{equation}
 0 \leq \widetilde{S}_T(t) - S_T(t) < \sup_{0\leq s \leq t} ( \widetilde{S}_T(s) + R(s)- L(s) ) + d_{max} .
 \end{equation}
 \end{lem}
 
 \begin{proof}
 Clearly, $V_{k,l}(t) \leq \widetilde{V}_{k,l}(t)$, and thus $S_T(t) \leq \widetilde{S}_T(t)$. Therefore, we have that 
 $\widetilde{S}_T(t) -  S_T(t) \geq 0$ and the difference increases only when C\ref{c1} is performed. Suppose that happens at time $t$ 
 and 
 a sleeping vertex of degree $j>0$ gets activated, then C\ref{c2} applies immediately and we have $A_T(t) \leq j-1 < d_{max}$, and 
 consequently,
 \begin{align*}
  \widetilde{S}_T(t) - S_T(t)  &= \widetilde{S}_T(t) - ( L(t) - R(t) - A_T(t) ) \\
 &< \widetilde{S}_T(t)+R(t) -L(t)+d_{max} .
 \end{align*}
 Since $\widetilde{S}_T(t) -  S_T(t)$ does not change in the intervals during which C\ref{c1} is not performed, 
 $\widetilde{S}_T(t) - S_T(t) \leq  \widetilde{S}_T(s) - S_T(s)$, where $s$ is the last time before $t$ that C\ref{c1} was performed. 
 The lemma follows.
 \end{proof}
 
 Let 
 \begin{equation}
 \widetilde{A}_T(t) := L(t)-R(t)-\widetilde{S}_T(t) = A_T(t) -  ( \widetilde{S}_T(t) - S_T(t) ).
 \end{equation}
 
 Then, Lemma \ref{lem4} can be rewritten as
 \begin{equation}
 \label{at2}
 \widetilde{A}_T(t) \leq A_T(t) < \widetilde{A}_T(t)-\inf_{s\leq t} \widetilde{A}_T(s) + d_{max}.
 \end{equation}
 
 Also, by Lemmas \ref{lem1}, \ref{lem2} and \ref{lem3} and (\ref{hx}),
 \begin{equation}
 \label{at}
  \sup_{t\geq 0}\left| n^{-1}\widetilde{A}_T(t)-H(e^{-t})\right|\xrightarrow{p} 0 .
 \end{equation}
 
 \begin{lem}
 \label{lem5}
  Suppose that Condition~\ref{degden} holds and let $H(x)$ be given by (\ref{hx}). 
 \begin{enumerate}[(i)]
 \item If $\mathbb{E}[D^{(t)}D] > \mathbb{E}[D^{(t)}+D]$, then there is a unique $\xi \in (0,1)$, such that $H(\xi)=0$; moreover, 
 $H(x)<0$ 
 for $x \in (0,\xi)$ and $H(x)>0$ for $x \in (\xi,1)$.
 \item If $\mathbb{E}[D^{(t)}D] \leq \mathbb{E}[D^{(t)}+D]$, then $H(x)<0$ for $x \in (0,1)$.
 \end{enumerate}
 \end{lem}
 
 \begin{proof}
 Remark that $H(0)=H(1)=0$ and $H'(1)=2\mathbb{E}[D]-\mathbb{E}[D^{(r)}]-\mathbb{E}[D^{(t)}D]=
 \mathbb{E}[D+D^{(t)}]-\mathbb{E}[D^{(t)}D]$. 
 Furthermore we define $\phi(x):=H(x)/x = \lambda x- \lambda_r - \sum_{k,l}lp_{k,l}x^{k+l-1}$, which is a concave function on $(0,1]$, 
 in fact, 
 strictly concave unless $p_{k,l}=0$ whenever $k+l\geq 3$ and $l\geq 1$, in which case $H'(1)=p_{0,1}+p_{1,1}+\sum_{k\geq 1}kp_{k,0}\geq p_{0,1}+p_{1,0}=\mathbb{P}(D=1)>0$, by Condition \ref{degden}(iv).
 
 In case (ii), we thus have $\phi$ concave and $\phi'(1)=H'(1)-H(1)\geq 0$, with either the concavity or the above inequality strict, and thus 
 $\phi'(x)>0$ for all $x\in (0,1)$, whence $\phi(x)<\phi(1)=0$ for $x\in (0,1)$.
 
 In case (i), $H'(1)<0$, and thus $H(x)>0$ for $x$ close to $1$. Further,
 \begin{align*}
  H'(0)&=-\lambda_r-\sum_{\{(k,l):k+l=1\}}lp_{k,l} \\
 &= -\lambda_r-p_{0,1} \\
 &\leq -p_{1,0}-p_{0,1}< 0
 \end{align*}
by Condition \ref{degden}(iv), which implies that $H(x)<0$ for $x$ close to $0$. Hence there is at least one $\xi \in (0,1)$
 with $H(\xi)=0$. Now, since $H(x)/x$ is strictly concave and also $\phi(1)=H(1)=0$, there is at most one such $\xi$. This proves the result.
 \end{proof}

 \begin{proof}[Proof of Theorem~\ref{inf1out}]
 Let $\xi$ be the zero of H given by Lemma \ref{lem5}(i) and let $\tau:=-\ln\xi$. Then, by Lemma \ref{lem5}, $H(e^{-t})>0$ for 
 $0<t<\tau$, 
 and thus $\inf_{t\leq \tau}H(e^{-t})=0$. Consequently, (\ref{at}) implies
 \begin{equation}
 \label{at3}
 n^{-1}\inf_{t\leq \tau}\widetilde{A}_T(t) = n^{-1}\inf_{t\leq \tau}\widetilde{A}_T(t)-\inf_{t\leq \tau}H(e^{-t}) \xrightarrow{p} 0.
 \end{equation}
 
 Further, by Condition \ref{degden}(iii), $d_{max}=O(n^{1/2})$, and thus $n^{-1}d_{max} \to 0$. Consequently, by (\ref{at2}) and 
 (\ref{at3})
 \begin{equation}
 \label{at5}
 \sup_{t\leq \tau} n^{-1}\left|A_T(t)-\widetilde{A}_T(t)\right| = \sup_{t\leq \tau} n^{-1}\left|\widetilde{S}_T(t) - S_T(t)\right|
 \xrightarrow{p} 0,
 \end{equation}
 and thus, by (\ref{at}), 
 \begin{equation}
 \label{at4}
  \sup_{t\geq 0}\left| n^{-1}A_T(t)-H(e^{-t})\right|\xrightarrow{p} 0. 
 \end{equation}
 
 Let $0<\epsilon<\tau/2$. Since $H(e^{-t})>0$ on the compact interval $[\epsilon,\tau-\epsilon]$, (\ref{at4}) implies that whp $A_T(t)$ 
 remains 
 positive on $[\epsilon,\tau-\epsilon]$, and thus C\ref{c1} is not performed during this interval.
 
 On the other hand, again by Lemma \ref{lem5}(i), $H(e^{-\tau-\epsilon})<0$ and (\ref{at}) implies $n^{-1}\widetilde{A}_T(\tau+\epsilon) 
 \xrightarrow{p} H(e^{-\tau-\epsilon})$, while $A_T(\tau+\epsilon)\geq 0$. Thus, with $\delta := \left|H(e^{-\tau-\epsilon})\right|/2>0$,
 whp
 \begin{equation}
 \label{tau1}
 \widetilde{S}_T(\tau+\epsilon) -  S_T(\tau+\epsilon) =A_T(\tau+\epsilon)-\widetilde{A}_T(\tau+\epsilon) 
 \geq -\widetilde{A}_T(\tau+\epsilon)>n\delta ,
 \end{equation}
 while (\ref{at5}) implies that $ \widetilde{S}_T(\tau) - S_T(\tau)<n\delta$ whp. Consequently, whp $ \widetilde{S}_T(\tau+\epsilon) 
 - S_T(\tau+\epsilon)>\widetilde{S}_T(\tau) -  S_T(\tau)$, so C\ref{c1} is performed between $\tau$ and $\tau+\epsilon$.
 
 Let $T_1$ be the last time that C\ref{c1} is performed before $\tau/2$, let $x_n$ be the sleeping vertex declared active at this point 
 of 
 time and let $T_2$ be the next time C\ref{c1} is performed. We have shown that for any $\epsilon>0$, whp $0\leq T_1 \leq \epsilon$ and 
 $\tau-\epsilon \leq T_2 \leq \tau+\epsilon$; in other words, $T_1\xrightarrow{p} 0$ and $T_2\xrightarrow{p} \tau$.
 
 We next use the following lemma.
 
 \begin{lem}
 \label{lem6}
 Let $T_1^*$ and $T_2^*$ be two (random) times when C\ref{c1} are performed, with $T_1^* \leq T_2^*$, and assume that $T_1^* 
 \xrightarrow{p} t_1$ 
 and $T_2^*\xrightarrow{p} t_2$ where $0 \leq t_1 \leq t_2 \leq \tau$. If $C$ is the union of all the vertices informed between $T_1^*$ 
 and $T_2^*$, 
 then
 \begin{equation}
 \label{lem62}
 \left|C\right|/n \xrightarrow{p} g(e^{-t_1},e^{-t_1})-g(e^{-t_2},e^{-t_2}).
 \end{equation}
 \end{lem}
 
 \begin{proof}
 For all $t\geq 0$, we have 
 \begin{align*}
 \sum_{i,j} (\widetilde{V}_{i,j}(t) - V_{i,j}(t)) &\leq \sum_{i,j} j(\widetilde{V}_{i,j}(t) - V_{i,j}(t)) = \widetilde{S}_T(t)-S_T(t).
 \end{align*}
 Thus,
 \begin{align*}
 \left|C\right|&=\sum(V_{k,l}(T_1^* -)-V_{k,l}(T_2^* -))=\sum(\widetilde{V}_{k,l}(T_1^* -)-\widetilde{V}_{k,l}(T_2^* -)) + o_p(n) \\
 &= ng(e^{-T_1^*},e^{-T_1^*}) - ng(e^{-T_2^*},e^{-T_2^*}) + o_p(n).
 \end{align*}
 \end{proof}
 
 Let $C'$ be the set of vertices informed up till $T_1$ and $C''$ be the set of vertices informed between $T_1$ and $T_2$. Then, by 
 Lemma \ref{lem6}, we have that
 \begin{equation}
 \label{cdash}
 \frac{\left|C'\right|}{n} \xrightarrow{p} 0
 \end{equation}
 and
 \begin{equation}
 \label{cddash}
 \frac{\left|C''\right|}{n} \xrightarrow{p} g(1,1) - g(e^{-\tau},e^{-\tau}) = 1 - g(e^{-\tau},e^{-\tau}).
 \end{equation}
 
 Evidently, $C'' \subset C(x_n)$.
 Note that $C(x_n)=\left\{y \in v(G^*(n,(d_i)_1^n)) : x_n\to y\right\}$. It is clear that if $x_n\to y$, then 
 $y \notin \left(C'\cup C''\right)^c$. 
 Therefore, we have that $C(x_n)\subset C'\cup C''$, which implies that
 \begin{equation}
 \label{sdw}
 \left|C''\right| \leq \left|C(x_n)\right| \leq \left|C'\right|+\left|C''\right|,
 \end{equation}
 and thus, from (\ref{cdash}) and (\ref{cddash}),
 \begin{equation}
  \frac{\left|C(x_n)\right|}{n} \xrightarrow{p} 1-g(e^{-\tau},e^{-\tau}),
 \end{equation}
 which completes the proof of Theorem \ref{inf1out}.

 \end{proof}

 \begin{proof}[Proof of Theorem~\ref{inf2out}]
  We continue from where we left in the proof of previous theorem, with the following Lemmas. Assumptions of Theorem \ref{inf1out} continue to hold for 
  what follows in this section.
  \begin{lem}
  \label{lem8}
  $\forall \epsilon>0$, let 
  \begin{equation*}
  \mathbb{A}(\epsilon):=\left\{y \in v(G^*(n,(d_i)_1^n)) :\frac{\left|C(y)\right|}{n}\geq \epsilon \text{ and }\left|\frac{\left|C(y)\right|}{n} - 
  (1-g(\xi,\xi))\right|\geq \epsilon \right\}.
  \end{equation*} 
  Then,
  \begin{equation}
  \forall \epsilon, \quad \frac{\left|\mathbb{A}(\epsilon)\right|}{n} \xrightarrow{p} 0.
  \end{equation}
  \end{lem}
  \begin{proof}
  Suppose the converse is true. Then, there exists $\delta>0$, $\delta'>0$ and a sequence $(n_k)_{k>0}$ such that
  \begin{equation}
  \forall k, \quad \mathbb{P}\left(\frac{\left|\mathbb{A}(\epsilon)\right|}{n_k}>\delta \right)>\delta'.
  \end{equation}
  Since the vertex initially informed to start the transmission process, say $a$, is uniformly chosen, we have 
  \begin{equation}
  \forall n_k, \quad \mathbb{P}(a\in \mathbb{A}(\epsilon))>\delta \delta'
  \end{equation}
  and thus,
  \begin{equation}
  \forall k, \quad \mathbb{P}\left(\frac{\left|C'\right|}{n_k}\geq \epsilon \text{ or } \left|\frac{\left|C''\right|}{n_k} - 
  (1-g(\xi,\xi))\right|\geq \epsilon \right)>\delta \delta',
  \end{equation}
  which contradicts (\ref{cddash}).
  \end{proof}

  \begin{lem}
  \label{lem9}
  For every $\epsilon>0$, let 
  \begin{equation}
  \mathbb{B}(\epsilon):=\left\{y \in C'\cup C'': \left|C(y)\right|/n \geq \epsilon \text{ and } \left|C(y)\vartriangle C^* \right|/n 
 \geq \epsilon\right\}.
  \end{equation}
  Then,
  \begin{equation}
  \label{bep}
   \forall \epsilon, \quad  \frac{\left|\mathbb{B}(\epsilon)\right|}{n} \xrightarrow{p} 0.
  \end{equation}
  \end{lem}
  \begin{proof}
  Recall that for any three sets $A$, $B$ and $C$, we have that $A\vartriangle B \subset (A\vartriangle C) \cup (B \vartriangle C)$. 
 Therefore, for any $y\in C'\cup C''$, we have that
  \begin{equation}
  C(y)\vartriangle C^* \subset \left[C(y)\vartriangle (C'\cup C'')\right] \cup \left[C^* \vartriangle (C'\cup C'')\right].
  \end{equation}
  But recall that $C^* \subset C'\cup C''$ and by a similar argument, for every $y\in C'\cup C''$, $C(y)\subset C'\cup C''$. Thus,
  \begin{equation}
  C(y)\vartriangle C^* \subset \left[(C'\cup C'')\setminus C(y)\right] \cup \left[(C'\cup C'') \setminus C^* \right] .
  \end{equation}
  Hence, if $\left| C(y)\vartriangle C^*\right|/n \geq \epsilon$, then either $\left|(C'\cup C'')\setminus C(y)\right|/n 
 \geq \epsilon/2$ or $\left|(C'\cup C'') \setminus C^* \right|/n \geq \epsilon/2$. Consequently, 
  \begin{align*}   
  \mathbb{B}(\epsilon)\subset &\left\{y \in v(G^*(n,(d_i)_1^n)): \epsilon \leq \left|C(y)\right|/n \leq \left|(C'\cup C'')\right|/n 
 - \epsilon/2\right\}\\
  &\cup \left\{y \in v(G^*(n,(d_i)_1^n)): \left|(C'\cup C'') \setminus C^* \right|/n \geq \epsilon/2 \right\}.
  \end{align*}
  Letting $e1:=\left|\left\{y \in v(G^*(n,(d_i)_1^n)): \epsilon \leq \left|C(y)\right|/n \leq \left|(C'\cup C'')\right|/n 
 - \epsilon/2\right\}\right|/n$ and $E2:=\left\{\left|(C'\cup C'') \setminus C^* \right|/n \geq \epsilon/2\right\}$, we have
  \begin{equation}
  \mathbb{B}(\epsilon)/n \leq e1+\mathbf{1}_{E2}.
  \end{equation}
  Now, $e1 \xrightarrow{p} 0$ by (\ref{cddash}) and Lemma \ref{lem8}, while $\mathbf{1}_{E2} \xrightarrow{p} 0$ because $\mathbb{P}(E2) 
  \to 0$ 
 by (\ref{cdash}), (\ref{cddash}) and (\ref{sdw}).
  This concludes the proof.
  \end{proof}
  
  \begin{lem}
  \label{lem10}
  Let $T_3$ be the first time after $T_2$ that C\ref{c1} is performed and let $z_n$ be the sleeping vertex activated at this moment. 
 If $C'''$ is the set of vertices informed between $T_2$ and $T_3$, then
  \begin{equation}
  \label{cdddash}
  \frac{\left|C'''\right|}{n} \xrightarrow{p} 0.
  \end{equation}
  \end{lem}
  \begin{proof}
  Since $\widetilde{S}_T(t)-S_T(t)$ increases by at most $d_{max}=o_p(n)$ each time C\ref{c1} is performed, we obtain that
  \begin{equation}
  \sup_{t\leq T_3}(\widetilde{S}_T(t)-S_T(t))\leq \sup_{t\leq T_2}(\widetilde{S}_T(t)-S_T(t)) + d_{max}=o_p(n).
  \end{equation}
  Comparing this to (\ref{tau1}) we see that for every $\epsilon>0$, whp $\tau+\epsilon > T_3$. Since also $T_3 > T_2 \xrightarrow{p} 
  \tau$, 
 it follows that $T_3 \xrightarrow{p} \tau$. This in combination with Lemma \ref{lem6} yields that
  \begin{equation*}
  \frac{\left|C'''\right|}{n} \xrightarrow{p} 0.
 \end{equation*}
  \end{proof}

  \begin{lem}
  \label{lem11}
  For every $\epsilon>0$, let 
  \begin{equation}
  \label{czn2}
  \mathbb{C}(\epsilon):=\left\{z \in \left(C'\cup C''\right)^c : \left|C(z)\right|/n \geq \epsilon \text{ and } 
 \left|C(z)\vartriangle C^* \right|/n \geq \epsilon\right\}.
  \end{equation} 
  Then, we have that
  \begin{equation}
  \label{czn3} 
  \forall \epsilon, \quad \frac{\left|\mathbb{C}(\epsilon)\right|}{n} \xrightarrow{p} 0.
  \end{equation}
  \end{lem}
  \begin{proof}
  We start by remarking that by Lemma \ref{lem8}, it is sufficient to prove that 
  \begin{equation}
  \frac{\left|\mathbb{C}(\epsilon)\cap \mathbb{A}^c(\epsilon)\right|}{n} \xrightarrow{p} 0.
  \end{equation}
  
  Now assume that there exist $\delta,\delta'>0$ and a sequence $(n_k)_{k>0}$ such that
  \begin{equation}
  \forall k, \quad \mathbb{P}\left(\frac{\left|\mathbb{C}(\epsilon)\cap \mathbb{A}^c(\epsilon)\right|}{n_k} >\delta \right) > \delta'.
  \end{equation}
  Let 
  \begin{equation*}
   \mathcal{E}_1:=\{\text{Configuration Model completely revealed}\},
  \end{equation*}
  
  \begin{equation*}
   \mathcal{E}_2:=\{\text{Influence propagation revealed upto }C''\}
  \end{equation*}

   and $\mathcal{E}_3:=\mathcal{E}_1\cap \mathcal{E}_2$.Then, denoting be $z_{n_k}$ the vertex awakened by C\ref{c1} at time $T_2$, we have that
  \begin{align*}
  &\mathbb{P}\left(z_{n_k} \in \mathbb{C}(\epsilon)\cap \mathbb{A}^c(\epsilon)\middle|\mathcal{E}_3\right) \\
  &\geq \frac{\left|\mathbb{C}(\epsilon)\cap \mathbb{A}^c(\epsilon)\right|}{n_k-\left|C'\cup C''\right|} 
  \mathbf{1}\left(\frac{\left|\mathbb{C}(\epsilon)\cap \mathbb{A}^c(\epsilon)\right|}{n_k} >\delta \right)\\
  &\geq \frac{\left|\mathbb{C}(\epsilon)\cap \mathbb{A}^c(\epsilon)\right|}{n_k} 
  \mathbf{1}\left(\frac{\left|\mathbb{C}(\epsilon)\cap \mathbb{A}^c(\epsilon)\right|}{n_k} >\delta \right)\\
  &\geq \delta \mathbf{1}\left(\frac{\left|\mathbb{C}(\epsilon)\cap \mathbb{A}^c(\epsilon)\right|}{n_k} >\delta \right).
  \end{align*}
  Taking expectations, we have
  \begin{equation}
   \mathbb{P}\left(z_{n_k} \in \mathbb{C}(\epsilon)\cap \mathbb{A}^c(\epsilon)\right) \geq \delta \delta'.
  \end{equation}

  But this leads to contradiction. Indeed, we have that
  \begin{equation}
  C(z_n)\vartriangle C^* \subset \left[C(z_n)\vartriangle (C'\cup C''\cup C''')\right] \cup \left[C^* \vartriangle 
  (C'\cup C''\cup C''')\right].
  \end{equation}
  Again recall that $C^* \subset C'\cup C'' \cup C'''$ and by a similar argument, $C(z_n)\subset C'\cup C''\cup C'''$ so that
  \begin{equation}
  C(z_n)\vartriangle C^* \subset \left[(C'\cup C''\cup C''')\setminus C(z_n)\right] \cup \left[(C'\cup C''\cup C''') \setminus 
  C^* \right] .
  \end{equation}
  Hence, if $\left| C(z_n)\vartriangle C^*\right|/n \geq \epsilon$, then either 
  \begin{align*}
 & \left|(C'\cup C''\cup C''')\setminus C(z_n)\right|/n \geq \epsilon/2 \\
  \text{equivalently, }&\left|C(z_n)\right|/n \leq \left|(C'\cup C''\cup C''')\right|/n - \epsilon/2,
  \end{align*}
   or, 
   \begin{equation*}
   \left|(C'\cup C''\cup C''') \setminus C^* \right|/n \geq \epsilon/2.
   \end{equation*}
 Let 
 \begin{equation*}
 E3:=\left\{\left|C(z_n)\right|/n \leq \left|(C'\cup C''\cup C''')\right|/n - \epsilon/2 \right\}
 \end{equation*}
  and 
  \begin{equation*}
  E4:=\left\{\left|(C'\cup C''\cup C''') \setminus C^* \right|/n \geq \epsilon/2\right\}.
  \end{equation*}
 Now assume that $z_{n} \in \mathbb{C}(\epsilon)\cap \mathbb{A}^c(\epsilon)$. This implies that either $E4$ holds or $\left\{1-g(\xi,\xi)
 -\epsilon \leq \frac{\left|C(z_n)\right|}{n} \leq 1-g(\xi,\xi)+\epsilon \right\}\cap E3$ holds. But thanks to (\ref{cdash}), 
 (\ref{cddash}) 
 and Lemma \ref{lem10}, neither of these two events hold with asymptotically positive probability.
  
  This completes the proof.
  \end{proof}
  
  Finally, Lemma \ref{lem9} and Lemma \ref{lem11} allow us to conclude that
  \begin{equation}
  \forall \epsilon, \quad  \frac{\left|\mathbb{C}^s(\epsilon)\right|+\left|\mathbb{C}^L(\epsilon)\right|}{n} \xrightarrow{p} 1.
  \end{equation}
  
 \end{proof}

 \section{Analysis of the Dual Back-Propagation Process}
 \label{s.Reverse}
 Now we introduce the algorithm to trace the possible sources of influence of a randomly chosen vertex.
 We borrow the terminology from the previous section, only in this case we put a \textit{bar} over the label to indicate that we're 
 talking about 
 the dual process. The analysis also proceeds along the same lines as that of the original process, and we do not give the proof when 
 it differs 
 from the analogous proof in the previous section only by notation. 
 
 As before, we initially give all the half-edges i.i.d. random maximal lifetimes with distribution 
 $\overline{\tau} \sim \text{exp}(1)$ and then go through the following algorithm.
 
 \begin{enumerate}[$\overline{\text{C}}$1]
  \item If there is no active half-edge (as in the beginning), select a sleeping vertex and declare it active, along with all its 
  half-edges.  For definiteness, we choose the vertex uniformly at random among all sleeping vertices. If there is no sleeping 
 vertex left, the process stops. \label{cb1}
  \item Pick an active half-edge and kill it. \label{cb2}
  \item Wait until the next transmitter half-edge dies (spontaneously). This is joined to the one killed in previous step to form an 
  edge of 
 the graph. 
 	If the vertex it belongs to is sleeping, we change its status to active, along with all of its half-edges. Repeat from the first 
 step. \label{cb3}
 \end{enumerate}
 
 Again, as before, $\overline{L}(0)=2m-1$ and we have the following consequences of \textit{Glivenko-Cantelli} theorem.
  
 \begin{lem}
 \label{lem1b}
  As $n\to \infty$ ,
 \begin{equation}
  \sup_{t\geq 0}\left| n^{-1}\overline{L}(t)-\lambda e^{-2t}\right| \xrightarrow{p} 0.
 \end{equation}
 \end{lem}

 Let $\overline{V}_{k,l}(t)$ be the number of sleeping vertices at time $t$ which had receiver and transmitter degrees $k$ and $l$ 
 respectively 
 at time $0$. It is easy to see that
 \begin{equation}
  \overline{S}(t)=\sum_{k,l}(ke^{-t}+l) \overline{V}_{k,l}(t).
 \end{equation}
 Let $\widetilde{\overline{V}}_{k,l}(t)$ be the corresponding number if the impact of $\overline{\text{C}}$\ref{cb1} on sleeping 
 vertices is ignored. 
 Correspondingly, let $\widetilde{\overline{S}}(t)=\sum_{k,l}(ke^{-t}+l) \widetilde{\overline{V}}_{k,l}(t)$.
 
 Then,
 \begin{lem}
 \label {lem3b}
   As $n\to \infty$ ,
 \begin{equation}
 \label{lem3b1}
  \sup_{t\geq 0}\left| n^{-1}\widetilde{\overline{V}}_{k,l}(t)-p_{k,l}e^{-lt})\right|\xrightarrow{p} 0 .
 \end{equation}
 for all $(k,l) \in \mathbb{N}^2$, and 
 \begin{equation}
 \label{lem3b2}
  \sup_{t\geq 0}\left| n^{-1} \sum_{k,l}\widetilde{\overline{V}}_{k,l}(t)-\overline{g}(e^{-t})\right| \xrightarrow{p} 0 .
 \end{equation}
 \begin{equation}
 \label{lem3b3}
  \sup_{t\geq 0}\left|n^{-1}\widetilde{\overline{S}}(t)-\overline{h}(e^{-t})\right| \xrightarrow{p} 0. 
 \end{equation} 
 \end{lem}
 
 \begin{proof}
  Again, (\ref{lem3b1}) follows from \textit{Glivenko-Cantelli} theorem. 
 
 To prove (\ref{lem3b3}), note that by (3) of Condition(\ref{degden}), $D_n=D_n^{(r)}+D_n^{(t)}$ are uniformly integrable, i.e., for 
 every 
 $\epsilon>0$ 
 there exists $K<\infty$ such that
 for all $n$, 
 \begin{equation}
  \mathbb{E}(D_n;D_n>K) = \sum_{(k,l;k+l>K)}(k+l)\frac{u_{k,l}}{n}<\epsilon.
 \end{equation}
 This, by Fatou's inequality, further implies that 
 \begin{equation}
  \sum_{(k,l;k+l>K)}(k+l)p_{k,l}<\epsilon.
 \end{equation}
 Thus, by (\ref{lem3b1}), we have whp,
 \begin{align*}
  \sup_{t\geq 0}\left| n^{-1}\widetilde{\overline{S}}(t)-\overline{h}(e^{-t})\right| &=  \sup_{t\geq 0}\left| \sum_{k,l}(ke^{-t}+l)
  (n^{-1} 
 \widetilde{\overline{V}}_{k,l}(t)-p_{k,l}e^{-lt})\right| \\
 &\leq \sum_{(k,l;k+l \leq K)} (k+l) \sup_{t\geq 0}\left| (n^{-1} \widetilde{\overline{V}}_{k,l}(t)-p_{k,l}e^{-lt})\right| + \\
 &\sum_{(k,l;k+l>K)}(k+l)(\frac{u_{k,l}}{n}+p_{k,l}) \\
 &\leq \epsilon +\epsilon + \epsilon,
 \end{align*}
 which proves (\ref{lem3b3}). A similar argument also proves (\ref{lem3b2}).
 \end{proof}
 
 \begin{lem}
 \label{lem4b}
  If $d_{max} := \max_id_i$ is the maximum degree of $G^*(n,(d_i)_1^n)$, then
 \begin{equation}
 0 \leq \widetilde{\overline{S}}(t) - \overline{S}(t) < \sup_{0\leq s \leq t} ( \widetilde{\overline{S}}(s) - L(s) ) + d_{max}.
 \end{equation}
 \end{lem}
 \begin{proof}
 Clearly, $\overline{V}_{k,l}(t) \leq \widetilde{\overline{V}}_{k,l}(t)$, and thus $\overline{S}(t) \leq \widetilde{\overline{S}}(t)$.
 Therefore, we have that $\widetilde{\overline{S}}(t) -  \overline{S}(t) \geq 0$ and the difference increases only when 
 $\overline{\text{C}}$\ref{cb1} is performed. 
 Suppose that happens at time $t$ and a sleeping vertex of degree $j>0$ gets activated, then $\overline{\text{C}}$\ref{cb2} applies 
 immediately and 
 we have $\overline{A}(t) \leq j-1 < d_{max}$, and consequently,
 \begin{align*}
  \widetilde{\overline{S}}(t) - \overline{S}(t)  &= \widetilde{\overline{S}}(t) - ( \overline{L}(t)  - \overline{A}(t) ) \\
 &< \widetilde{\overline{S}}(t) -\overline{L}(t)+d_{max}.
 \end{align*}
 Since $\widetilde{\overline{S}}(t) -  \overline{S}(t)$ does not change in the intervals during which $\overline{\text{C}}$\ref{cb1} is 
 not performed, 
 $\widetilde{\overline{S}}(t) - \overline{S}(t) \leq  \widetilde{\overline{S}}(s) - \overline{S}(s)$, where $s$ is the last time 
 before $t$ 
 that $\overline{\text{C}}$\ref{cb1} was performed. The lemma follows.
 \end{proof}
 
 Let 
 \begin{equation}
 \widetilde{\overline{A}}(t) := \overline{L}(t)-\widetilde{\overline{S}}(t) = \overline{A}(t) -  
 ( \widetilde{\overline{S}}(t) - \overline{S}(t) ).
 \end{equation}
 
 Then, Lemma \ref{lem4b} can be rewritten as
 \begin{equation}
 \label{at2b}
 \widetilde{\overline{A}}(t) \leq \overline{A}(t) < \widetilde{\overline{A}}(t)-\inf_{s\leq t} \widetilde{\overline{A}}(s) + d_{max}.
 \end{equation}
 
 Also, by Lemmas \ref{lem1b} and \ref{lem3b} and (\ref{Hx2}),
 \begin{equation}
 \label{atb}
  \sup_{t\geq 0}\left| n^{-1}\widetilde{\overline{A}}(t)-\overline{H}(e^{-t})\right|\xrightarrow{p} 0 .
 \end{equation}
 
 \begin{lem}
 \label{lem5b}
  Suppose that Condition~\ref{degden} holds and let $\overline{H}(x)$ be given by (\ref{Hx2}). 
 \begin{enumerate}[(i)]
 \item If $\mathbb{E}[D^{(t)}D] > \mathbb{E}[D^{(t)}+D]$, then there is a unique $\overline{\xi} \in (0,1)$, 
 such that $\overline{H}(\overline{\xi})=0$; moreover, $\overline{H}(x)<0$ for $x \in (0,\overline{\xi})$ and $\overline{H}(x)>0$ for 
 $x \in (\overline{\xi},1)$. \label{5bi}
 \item If $\mathbb{E}[D^{(t)}D] \leq \mathbb{E}[D^{(t)}+D]$, then $\overline{H}(x)<0$ for $x \in (0,1)$. \label{5bii}
 \end{enumerate}
 \end{lem}
 \begin{proof}
  Remark that $\overline{H}(0)=\overline{H}(1)=0$ 
 and $\overline{H}'(1)=2\mathbb{E}[D]-\mathbb{E}[(D^{(t)})^2]-\mathbb{E}[(D^{(r)})]-\mathbb{E}[D^{(r)}D^{(t)}]=\mathbb{E}[D+D^{(t)}]
 -\mathbb{E}[D^{(t)}D]$. 
 Furthermore we define $\overline{\phi}(x):=\overline{H}(x)/x = \lambda x- \sum_{k,l}lp_{k,l}x^{l-1}- \sum_{k,l}kp_{k,l}x^l$, which is a 
 concave function on $(0,1]$, 
 in fact, strictly concave unless $p_{k,l}=0$ whenever $l > 2$, or $l=2$ and $k\geq 1$ , 
 in which case $\overline{H}'(1)=\sum_{k\geq 0} p_{k,1}+\sum_{k\geq 0}k p_{k,0}\geq p_{1,0}+p_{0,1}>0$ by Condition \ref{degden}(iv).
 
 In case (\ref{5bii}), we thus have $\overline{\phi}$ concave and $\overline{\phi}'(1)=\overline{H}'(1)-\overline{H}(1)\geq 0$, with 
 either 
 the concavity 
 or the above inequality strict, and 
 thus $\overline{\phi}'(x)>0$ for all $x\in (0,1)$, whence $\overline{\phi}(x)<\overline{\phi}(1)=0$ for $x\in (0,1)$.
 
 In case (\ref{5bi}), $\overline{H}'(1)<0$, and thus $\overline{H}(x)>0$ for $x$ close to $1$. Further, in case (\ref{5bi}),
 \begin{equation}
  \overline{H}'(0)=-\sum_k p_{k,1}-\sum_k kp_{k,0}\leq -p_{1,0}-p_{0,1}<0 
 \end{equation}
 by Condition \ref{degden}(iv), which implies that $\overline{H}(x)<0$ for $x$ close to $0$. Hence there is at least one 
 $\overline{\xi} \in (0,1)$
 with $\overline{H}(\overline{\xi})=0$. Now, since $\overline{H}(x)/x$ is strictly concave and also $\overline{H}(1)=0$, 
 there is at most one such $\overline{\xi}$. This proves the result.

 \end{proof}
 
 \begin{proof}[Proof of Theorem~\ref{inf1bar}]
  Let $\overline{\xi}$ be the zero of $\overline{H}$ given by Lemma \ref{lem5b}(\ref{5bi}) and let $\overline{\tau}:=-\ln\overline{\xi}$. 
 Then, by Lemma \ref{lem5b},
  $\overline{H}(e^{-t})>0$ for $0<t<\overline{\tau}$, and thus $\inf_{t\leq \overline{\tau}}\overline{H}(e^{-t})=0$. Consequently, 
 (\ref{atb}) implies
 \begin{equation}
 \label{at3b}
 n^{-1}\inf_{t\leq \overline{\tau}}\widetilde{\overline{A}}(t) = n^{-1}\inf_{t\leq \overline{\tau}}\widetilde{\overline{A}}(t)
 -\inf_{t\leq \overline{\tau}}\overline{H}(e^{-t}) \xrightarrow{p} 0.
 \end{equation}
 Further, by Condition \ref{degden}(iii), $d_{max}=O(n^{1/2})$, and thus $n^{-1}d_{max} \to 0$. Consequently, by (\ref{at2b}) and 
 (\ref{at3b})
 \begin{equation}
 \label{at5b}
 \sup_{t\leq \overline{\tau}} n^{-1}\left|\overline{A}(t)-\widetilde{\overline{A}}(t)\right| 
 = \sup_{t\leq \overline{\tau}} n^{-1}\left|\widetilde{\overline{S}}(t) - \overline{S}(t)\right|\xrightarrow{p} 0
 \end{equation}
 and thus, by (\ref{atb}), 
 \begin{equation}
 \label{at4b}
  \sup_{t\geq 0}\left| n^{-1}\overline{A}(t)-\overline{H}(e^{-t})\right|\xrightarrow{p} 0 .
 \end{equation}
 
 Let $0<\epsilon<\overline{\tau}/2$. Since $\overline{H}(e^{-t})>0$ on the compact interval $[\epsilon,\overline{\tau}-\epsilon]$, 
 (\ref{at4b}) implies 
 that whp $\overline{A}(t)$ remains positive on $[\epsilon,\overline{\tau}-\epsilon]$, and thus $\overline{\text{C}}$\ref{cb1} is not 
 performed during this interval.
 
 On the other hand, again by Lemma \ref{lem5b}(\ref{5bi}), $\overline{H}(e^{-\overline{\tau}-\epsilon})<0$ and (\ref{atb}) implies 
 $n^{-1}\widetilde{\overline{A}}(\overline{\tau}+\epsilon) \xrightarrow{p} \overline{H}(e^{-\overline{\tau}-\epsilon})$, while 
 $\overline{A}(t)(\overline{\tau}+\epsilon)\geq 0$. Thus, with $\delta := \left|\overline{H}(e^{-\overline{\tau}-\epsilon})\right|/2>0$, 
 whp
 \begin{equation}
 \label{tau1b}
 \widetilde{\overline{S}}(\overline{\tau}+\epsilon) -  \overline{S}(\overline{\tau}+\epsilon) =\overline{A}(t)(\overline{\tau}+\epsilon)
 -\widetilde{\overline{A}}(\overline{\tau}+\epsilon) \geq -\widetilde{\overline{A}}(\overline{\tau}+\epsilon)>n\delta ,
 \end{equation}
 while (\ref{at5b}) implies that $ \widetilde{\overline{S}}(\overline{\tau}) - \overline{S}(\overline{\tau})<n\delta$ whp. Consequently, 
 whp 
 $ \widetilde{\overline{S}}(\overline{\tau}+\epsilon) - \overline{S}(\overline{\tau}+\epsilon)>\widetilde{\overline{S}}(\overline{\tau}) 
 -  \overline{S}(\overline{\tau})$, 
 so $\overline{\text{C}}$\ref{cb1} is performed between $\overline{\tau}$ and $\overline{\tau}+\epsilon$.
 
 Let $\overline{T}_1$ be the last time that $\overline{\text{C}}$\ref{cb1} is performed before $\overline{\tau}/2$, let $y_n$ be the 
 sleeping vertex declared active 
 at this point of time and let $\overline{T}_2$ be the next time $\overline{\text{C}}$\ref{cb1} is performed. We have shown that for 
 any $\epsilon>0$, whp 
 $0\leq \overline{T}_1 \leq \epsilon$ and $\overline{\tau}-\epsilon \leq \overline{T}_2 \leq \overline{\tau}+\epsilon$; in other words, 
 $\overline{T}_1\xrightarrow{p} 0$ and $\overline{T}_2\xrightarrow{p} \overline{\tau}$.
 
 We next use the following lamma.
 
 \begin{lem}
 \label{lem6b}
 Let $\overline{T}_1^*$ and $\overline{T}_2^*$ be two (random) times when $\overline{\text{C}}$\ref{cb1} are performed, with 
 $\overline{T}_1^* \leq \overline{T}_2^*$, 
 and assume that $\overline{T}_1^* \xrightarrow{p} t_1$ and $\overline{T}_2^*\xrightarrow{p} t_2$ where $0 \leq t_1 \leq t_2 \leq 
 \overline{\tau}$. 
 If $\overline{C}$ is the union of all the informer vertices reached between $\overline{T}_1^*$ and $\overline{T}_2^*$, then
 
 \begin{equation}
 \label{lem6b2}
 \left|\overline{C}\right|/n \xrightarrow{p} \overline{g}(e^{-t_1})-\overline{g}(e^{-t_2}).
 \end{equation}
 \end{lem}
 
 \begin{proof}

 For all $t\geq 0$, we have 
 \begin{align*}
 \sum_{i,j} (\widetilde{\overline{V}}_{i,j}(t) - \overline{V}_{i,j}(t)) &\leq \sum_{i,j} j(\widetilde{\overline{V}}_{i,j}(t) 
 - \overline{V}_{i,j}(t)) = \widetilde{\overline{S}}(t)-\overline{S}(t).
 \end{align*}
 Thus,
 \begin{align*}
 \left|\overline{C}\right|&=\sum(\overline{V}_{k,l}(\overline{T}_1^* -)-\overline{V}_{k,l}(\overline{T}_2^* -))
 =\sum(\widetilde{\overline{V}}_{k,l}(\overline{T}_1^* -)-\widetilde{\overline{V}}_{k,l}(\overline{T}_2^* -)) + o_p(n) \\
 &= n\overline{g}(e^{-\overline{T}_1^*}) - n\overline{g}(e^{-\overline{T}_2^*}) + o_p(n).
 \end{align*}
 
 \end{proof}
 
 Let $\overline{C}'$ be the set of possible influence sources traced up till $\overline{T}_1$ and $\overline{C}''$ be the set of those 
 traced between 
 $\overline{T}_1$ and $\overline{T}_2$. 
 Then, by Lemma \ref{lem6b}, we have that
 \begin{equation}
 \label{bcdash}
 \frac{\left|\overline{C}'\right|}{n} \xrightarrow{p} 0
 \end{equation}
 and
 \begin{equation}
 \label{bcddash}
 \frac{\left|\overline{C}''\right|}{n} \xrightarrow{p} \overline{g}(1) - \overline{g}(e^{-\overline{\tau}}) = 1 - 
 \overline{g}(e^{-\overline{\tau}}).
 \end{equation}
 
 Evidently, $\overline{C}'' \subset \overline{C}(y_n)$ and $\overline{C}(y_n)\subset \overline{C}'\cup \overline{C}''$, therefore
 \begin{equation}
 \label{sdwb}
 \left|\overline{C}''\right| \leq \left|\overline{C}(y_n)\right| \leq \left|\overline{C}'\right|+\left|\overline{C}''\right|
 \end{equation}
 and thus, from (\ref{bcdash}) and (\ref{bcddash}),
 \begin{equation}
  \frac{\left|\overline{C}(y_n)\right|}{n} \xrightarrow{p} 1-\overline{g}(e^{-\overline{\tau}}),
 \end{equation}
 which completes the proof.
 
 \end{proof}

 \begin{proof}[Proof of Theorem~\ref{inf2bar}]
  As in the previous section, we have the following set of Lemmmas, which we state without proof since the only change is notational. As before, 
  assumptions of Theorem \ref{inf1out} continue to hold.
  \begin{lem}
  \label{lem8b}
  $\forall \epsilon>0$, let 
  \begin{equation*}
  \overline{\mathbb{A}}(\epsilon):=\left\{x \in v(G^*(n,(d_i)_1^n)) :\frac{\left|\overline{C}(x)\right|}{n} \geq \epsilon \text{ and }\left|\frac{\left|\overline{C}(x)\right|}{n} - 
  (1-\overline{g}(\overline{\xi}))\right|\geq \epsilon 
 \right\}.
  \end{equation*} 
  Then,
  \begin{equation}
  \forall \epsilon, \quad \frac{\left|\overline{\mathbb{A}}(\epsilon)\right|}{n} \xrightarrow{p} 0.
  \end{equation}
  \end{lem}
  
  \begin{lem}
  \label{lem9b}
  For every $\epsilon>0$, let 
  \begin{equation}
  \overline{\mathbb{B}}(\epsilon):=\left\{x \in \overline{C}'\cup \overline{C}'': \left|\overline{C}(x)\right|/n 
 \geq \epsilon \text{ and } \left|\overline{C}(x)\vartriangle \overline{C}^* \right|/n \geq \epsilon\right\}.
  \end{equation}
  Then,
  \begin{equation}
  \label{bepb}
   \forall \epsilon, \quad  \frac{\left|\overline{\mathbb{B}}(\epsilon)\right|}{n} \xrightarrow{p} 0.
  \end{equation}
  \end{lem}
  
  \begin{lem}
  \label{lem10b}
  Let $\overline{T}_3$ be the first time after $\overline{T}_2$ that $\overline{\text{C}}$\ref{cb1} is performed and let $w_n$ be the 
 sleeping vertex activated 
 at this moment. If $\overline{C}'''$ is the set of informer vertices reached between $\overline{T}_2$ and $\overline{T}_3$, then
  \begin{equation}
  \label{bcdddash}
  \frac{\left|\overline{C}'''\right|}{n} \xrightarrow{p} 0.
  \end{equation}
  \end{lem}
  
  \begin{lem}
  \label{lem11b}
  For every $\epsilon>0$, let 
  \begin{equation}
  \label{cwn2}
  \overline{\mathbb{C}}(\epsilon):=\left\{w \in \left(\overline{C}'\cup \overline{C}''\right)^c : \left|\overline{C}(w)\right|/n 
 \geq \epsilon \text{ and } \left|\overline{C}(w)\vartriangle \overline{C}^* \right|/n \geq \epsilon\right\}.
  \end{equation} 
  Then, we have that
  \begin{equation}
  \label{cwn3} 
  \forall \epsilon, \quad \frac{\left|\overline{\mathbb{C}}(\epsilon)\right|}{n} \xrightarrow{p} 0.
  \end{equation}
  \end{lem}
  
  Finally, Lemma \ref{lem9b} and Lemma \ref{lem11b} allow us to conclude that
  \begin{equation}
  \forall \epsilon, \quad  \frac{\left|\overline{\mathbb{C}}^s(\epsilon)\right|+\left|\overline{\mathbb{C}}^L(\epsilon)\right|}{n} 
  \xrightarrow{p} 1.
  \end{equation}
 \end{proof} 
 
 \section{Duality Relation}
 \label{s.Duality}
 The forward and backward processes are linked through the tautology: $y \in C(x) \iff x\in \overline{C}(y)$. To prove the 
 Theorem \ref{dut}, 
 we consider the double sum: $\sum_{x,y \in v(G(n,(d_i)_1^n))}\mathbf{1}(y\in C(x))$.
 
 From here onwards, we abridge $v(G(n,(d_i)_1^n))$ to $v(G)$. Assumptions of Theorem \ref{inf1out} continue to hold throughout this section. 
 We start with the following Proposition. 
 
 \begin{prop}
 \label{dut1}
 We have,
 \begin{align*}
  \mathbf{A}_n &:=\left| n^{-2} \sum_{x,y \in v(G)}\mathbf{1}(y\in C(x))-n^{-2} \sum_{x,y \in v(G)}\mathbf{1}(x\in 
 \overline{C}^*)\mathbf{1}(y\in C(x)\cap C^*)\right| \\
 & \xrightarrow{p} 0 ,
 \end{align*}
 when $n \to \infty$.
 \end{prop}
 
 \begin{proof}
 The Proposition follows from the following two Lemmas.
 
 \begin{lem}
 \label{dut2}
 For any $\epsilon >0$ and $n\to \infty$,
 \begin{equation*}
 \left| n^{-2} \sum_{x,y \in v(G)}\mathbf{1}(y\in C(x))-n^{-2} \sum_{x,y \in v(G)}\mathbf{1}(y\in C(x)\cap C^*)\right| 
 \leq 2\epsilon + R_n^1(\epsilon) ,
 \end{equation*}
 where $R_n^1(\epsilon) \xrightarrow{p} 0$.
 \end{lem}
 
 \begin{proof}
 For $\epsilon>0$, we have
 \begin{align*}
  & \left| n^{-2} \sum_{x,y}\mathbf{1}(y\in C(x))-n^{-2} \sum_{x,y}\mathbf{1}(y\in C(x)\cap C^*)\right| \\
 \leq \quad & n^{-2} \sum_{x}\min (|C(x)|,|C(x)\vartriangle C^*|) \\
  = \quad & n^{-2} \sum_{x \in \mathbb{C}^s(\epsilon)}\min \left(|C(x)|,|C(x)\vartriangle C^*|\right)\\
 & + n^{-2} \sum_{x \in \mathbb{C}^L(\epsilon)} \min \left(|C(x)|,|C(x)\vartriangle C^*|\right)  \\
 & + n^{-2} \sum_{x \notin \mathbb{C}^s(\epsilon)\cup \mathbb{C}^L(\epsilon)} \min \left(|C(x)|,|C(x)\vartriangle C^*|\right) \\
  \leq \quad & n^{-1} \sum_{x \in \mathbb{C}^s(\epsilon)}\epsilon+n^{-1}\sum_{x \in \mathbb{C}^L(\epsilon)}\epsilon 
  +n^{-1}\sum_{x \notin \mathbb{C}^s(\epsilon)\cup \mathbb{C}^L(\epsilon)}1 \\
  \leq \quad & \epsilon + \epsilon + \left( 1 - \frac{|\mathbb{C}^s(\epsilon)| + |\mathbb{C}^L(\epsilon)|}{n} \right).
 \end{align*}
 Taking $R_n^1(\epsilon):=  1 - \frac{|\mathbb{C}^s(\epsilon)| + |\mathbb{C}^L(\epsilon)|}{n}$ and using Theorem \ref{inf2out}, we conclude the proof.
 \end{proof}
 
 \begin{lem}
 \label{dut3}
 For any $\epsilon >0$ and $n\to \infty$,
 \begin{align*}
 &\left| n^{-2} \sum_{x,y \in v(G)}\mathbf{1}(y\in C(x)\cap C^*)-n^{-2} \sum_{x,y \in v(G)}
 \mathbf{1}(x\in \overline{C}^*)\mathbf{1}(y\in C(x)\cap C^*)\right| \\
 &\leq 2\epsilon + R_n^2(\epsilon) ,
 \end{align*}
 where $R_n^2(\epsilon) \xrightarrow{p} 0$.
 \end{lem}
 
 \begin{proof}
 Since $y \in C(x) \iff x\in \overline{C}(y)$, we have
 \begin{equation}
  \sum_{x,y \in v(G)}\mathbf{1}(y\in C(x)\cap C^*) = \sum_{x,y \in v(G)}\mathbf{1}(x\in \overline{C}(y))
 \mathbf{1}(y\in C^*)
 \end{equation}
 and 
 \begin{equation}
  \sum_{x,y}\mathbf{1}(x\in \overline{C}^*)\mathbf{1}(y\in C(x)\cap C^*) = \sum_{x,y}\mathbf{1}(x\in \overline{C}(y)\cap \overline{C}^*)
 \mathbf{1}(y\in C^*).
 \end{equation}
 Consequently,
 \begin{align*}
  &\left| n^{-2} \sum_{x,y}\mathbf{1}(y\in C(x)\cap C^*)-n^{-2} \sum_{x,y}\mathbf{1}(x\in \overline{C}^*)
  \mathbf{1}(y\in C(x)\cap C^*)\right| \\
 &\leq  n^{-2} \sum_{y}\mathbf{1}(y\in C^*) \min \left(|\overline{C}(y)|,|\overline{C}(y)\vartriangle \overline{C}^*|\right) \\
 & \leq n^{-2} \sum_{y}\min \left(|\overline{C}(y)|,|\overline{C}(y)\vartriangle \overline{C}^*|\right).
 \end{align*}
 The result follows by the arguments similar to those in the proof of Lemma \ref{dut2}, with $R_n^2(\epsilon):=  1 - 
 \frac{|\overline{\mathbb{C}}^s(\epsilon)|
  + |\overline{\mathbb{C}}^L(\epsilon)|}{n}$.
 \end{proof}
 
 \end{proof}
 
 Next, we have the following two Propositions, which lead to Theorem \ref{dut}.
 
 \begin{prop}
 \label{dut4}
 For any $\epsilon >0$ and $n\to \infty$,
 \begin{equation}
 \left| n^{-1}|\mathbb{C}^L(\epsilon)|-n^{-1}|\mathbb{C}^L(\epsilon)\cap \overline{C}^*|\right| \leq \alpha^1 \epsilon + R_n^3(\epsilon) ,
 \end{equation}
 where $\alpha^1>0$ is a constant and $R_n^3(\epsilon) \xrightarrow{p} 0$. Analogously,
 \begin{equation}
 \label{dut6}
 \left| n^{-1}|\overline{\mathbb{C}}^L(\epsilon)|-n^{-1}|\overline{\mathbb{C}}^L(\epsilon)\cap C^*|\right| \leq \alpha^2 \epsilon + 
 R_n^4(\epsilon) 
 \end{equation}
 where $\alpha^2>0$ is a constant and $R_n^4(\epsilon) \xrightarrow{p} 0$.
 \end{prop}
 
 \begin{proof}
 
 Remark that
 \begin{align*}
 \sum_{x,y \in v(G)}\mathbf{1}(y\in C(x))=  & \sum_{x\in v(G), y \in \overline{\mathbb{C}}^L(\epsilon)}\mathbf{1}(x\in \overline{C}(y)) 
 +\sum_{x \in v(G), y\in \overline{\mathbb{C}}^s(\epsilon) }\mathbf{1}(x\in \overline{C}(y)) \\
  & +\sum_{x \in v(G), y \notin \overline{\mathbb{C}}^s(\epsilon)\cup \overline{\mathbb{C}}^L(\epsilon)}\mathbf{1}(x\in \overline{C}(y)).
 \end{align*}
 Therefore, using the arguments similar to those in the proof of Lemma \ref{dut2}, we have
 \begin{equation}
 \label{dut7}
 \left| n^{-2} \sum_{x,y \in v(G)}\mathbf{1}(y\in C(x))-n^{-2}|\overline{C}^*|.|\overline{\mathbb{C}}^L(\epsilon)|\right| \leq 2\epsilon 
 + R_n^2(\epsilon) .
 \end{equation}
 
 In the same way,
 \begin{equation*}
 \left| n^{-2} \sum_{x,y \in v(G)}\mathbf{1}(y\in C^*)\mathbf{1}(x\in \overline{C}(y)\cap \overline{C}^*)-n^{-2}|\overline{C}^*|.
 |\overline{\mathbb{C}}^L(\epsilon)\cap C^*|\right| \leq 2\epsilon + R_n^2(\epsilon). 
 \end{equation*}
 
 From the above two equations and using Proposition \ref{dut1}, we have
 \begin{equation*}
 \left| n^{-2} |\overline{C}^*|.|\overline{\mathbb{C}}^L(\epsilon)|-n^{-2}|\overline{C}^*|.|\overline{\mathbb{C}}^L(\epsilon)
 \cap C^*|\right| \leq 4\epsilon + 2R_n^2(\epsilon)+A_n .
 \end{equation*}
 Now using Theorem \ref{inf1out} and taking $\alpha^2:=\frac{5}{1-\overline{g}(\overline{\xi},\overline{\xi})}$ 
 and $R_n^4(\epsilon)= \frac{3R_n^2(\epsilon)+
 2A_n}{1-\overline{g}(\overline{\xi},\overline{\xi})}$, 
 we have the second part of the proposition. The proof of the first part is similar, with $\alpha^1:=\frac{5}{1-g(\xi,\xi)}$ 
 and $R_n^3(\epsilon)= \frac{3R_n^1(\epsilon)+
 2A_n}{1-g(\xi,\xi)}$.
 
 \end{proof}
 
 \begin{prop}
 \label{dut5}
 For any $\epsilon >0$,
 \begin{equation*}
 \left| n^{-2} \sum_{x,y \in v(G)}\mathbf{1}(y\in C(x))-n^{-2}|C^*\cap \overline{\mathbb{C}}^L(\epsilon)|.
 |\mathbb{C}^L(\epsilon)|\right| \leq 3 \epsilon + R_n^1(\epsilon) + R_n^2(\epsilon) 
 \end{equation*}
 \end{prop}
 
 \begin{proof}
 We can upper bound the double sum thus,
 \begin{align*}
 \sum_{x,y \in v(G)}\mathbf{1}(y\in C(x)) \leq  &\sum_{x\in \mathbb{C}^L(\epsilon),y \in \overline{\mathbb{C}}^L(\epsilon)}
 \mathbf{1}(y\in C(x))\\
 &+\sum_{x\in \mathbb{C}^s(\epsilon),y \in v(G)}\mathbf{1}(y\in C(x))\\
  &+\sum_{x\in v(G),y \in  \overline{\mathbb{C}}^s(\epsilon)}\mathbf{1}(y\in C(x))\\
  &+\sum_{x \notin \mathbb{C}^s(\epsilon)\cup \mathbb{C}^L(\epsilon), y\in v(G)}\mathbf{1}(y\in C(x)) \\
 &+\sum_{x \in v(G), y \notin \mathbb{C}^s(\epsilon)\cup \mathbb{C}^L(\epsilon)}\mathbf{1}(y\in C(x)).
 \end{align*}
 The result follows, once again, by using the arguments similar to those in the proof of Lemma \ref{dut2}. 
 
 \end{proof}
 
 Now, from Proposition \ref{dut5} and (\ref{dut6}) and (\ref{dut7}) from Proposition \ref{dut4}, we can conclude the proof of 
 Theorem \ref{dut}, with 
 $\alpha:=5+\alpha^1$ and $R_n(\epsilon):=R_n^1(\epsilon)+2R_n^2(\epsilon)+R_n^4(\epsilon)$. 
 
 The Corollary \ref{du} follows from Theorem \ref{dut} and Proposition \ref{dut4}.

\subsection*{Acknowledgements}
We thank Ren{\'e} Schott for introducing us to the influence propagation dynamic analysed in this paper through a pre-print of \cite{Comets} and thus motivating this study. We also thank Marc Lelarge for his useful suggestions regarding the analytical tools for exploration on Configuration Model and pointing us to \cite{JanLuc}, which has heavily influenced our approach.


\end{document}